\documentclass[11pt]{amsart}

\usepackage[a4paper,margin=1in]{geometry}
\usepackage[T1]{fontenc}
\usepackage[utf8]{inputenc}

\usepackage{amsmath,amssymb,amsfonts,amsthm}
\usepackage{mathtools}
\usepackage{mathrsfs}

\usepackage{enumitem}
\usepackage{aliascnt}
\usepackage[hidelinks]{hyperref}
\usepackage[nameinlink,noabbrev]{cleveref}

\newtheorem{theorem}{Theorem}[section]

\newaliascnt{proposition}{theorem}
\newtheorem{proposition}[proposition]{Proposition}
\aliascntresetthe{proposition}

\newaliascnt{lemma}{theorem}
\newtheorem{lemma}[lemma]{Lemma}
\aliascntresetthe{lemma}

\newaliascnt{corollary}{theorem}

\aliascntresetthe{corollary}

\newaliascnt{conjecture}{theorem}
\newtheorem{conjecture}[conjecture]{Conjecture}
\aliascntresetthe{conjecture}

\theoremstyle{definition}

\newaliascnt{definition}{theorem}

\aliascntresetthe{definition}

\newaliascnt{example}{theorem}

\aliascntresetthe{example}

\newaliascnt{remark}{theorem}
\newtheorem{remark}[remark]{Remark}
\aliascntresetthe{remark}

\newcommand{\R}{\mathbb{R}}

\newcommand{\Pcal}{\mathcal{P}}

\newcommand{\Law}{\operatorname{Law}}
\newcommand{\tr}{\operatorname{tr}}
\newcommand{\W}{W}
\newcommand{\Leb}{\lambda}
\newcommand{\ip}[2]{\left\langle #1,#2\right\rangle}
\newcommand{\qmu}{\mathsf q_{\mu}}

\newcommand{\mtwo}{m_2}
\newcommand{\MC}{\operatorname{MC}}
\newcommand{\Perg}{\operatorname{Per}_{\gamma_n}}

\newcommand{\ind}{\mathbf{1}}
\newcommand{\BVloc}{\operatorname{BV_{loc}}}
\title{Couplings Farthest from the Independent Gaussian}
\author{Stefan Schrott}
\date{\today}
\subjclass[2020]{Primary 49Q22; Secondary 60E15, 62H20}
\keywords{Optimal transport, Wasserstein distance, extremal couplings, prescribed marginals, Gaussian measures, Gaussian isoperimetric inequality, measures of dependence}

\begin{document}

\begin{abstract}
Motivated by Wasserstein measures of dependence, we study the largest possible 2-Wasserstein distance between a joint distribution and the product of its prescribed marginals. For two uniform marginals, Catalano and Lavenant conjectured that the monotone and antimonotone couplings maximize the distance from the independent coupling. We prove the Gaussian analogue for an arbitrary number $n\geq 2$ of one-dimensional standard Gaussian marginals. More generally, for every probability measure $\mu$ on $\mathbb R$ with finite second moment, we characterize the laws on $\R^n$ with all marginals equal to $\mu$ that are farthest from the $n$-dimensional standard Gaussian.

\end{abstract}

\maketitle

\section{Introduction}
Motivated by Wasserstein measures of dependence, we study the problem
of identifying couplings that are farthest, in 2-Wasserstein
distance, from the product of their marginals. For probability measures
$\mu$ and $\nu$ on $\R^k$ with finite second moments, recall that their
squared 2-Wasserstein distance is given by
\[
  \W_2^2(\mu,\nu)
  :=
  \inf_{\kappa\in\Pi(\mu,\nu)}
  \int_{\R^k\times\R^k}|x-y|^2\,d\kappa(x,y),
\]
where $\Pi(\mu,\nu)$ denotes the set of couplings of $\mu$ and $\nu$,
that is, the probability measures on $\R^k\times\R^k$ with first
marginal $\mu$ and second marginal $\nu$. We write $\Pcal_2(\R^k)$ for the set of Borel probability measures on
$\R^k$ with finite second moment.

Let $X_1$ and $X_2$ be real-valued random variables with common law
$\mu$, and let
\[
  \pi=\Law(X_1,X_2)\in\Pi(\mu,\mu)
\]
denote their joint law. A natural way to quantify their dependence is
to consider the Wasserstein distance from $\pi$ to the independent
coupling $\mu\otimes\mu$, see e.g.\ \cite{CaLa25,MoSe22}. Since $\W_2$ is a
distance, the functional
\[
  \pi\longmapsto\W_2(\pi,\mu\otimes\mu),
  \qquad \pi\in\Pi(\mu,\mu),
\]
vanishes precisely at $\pi=\mu\otimes\mu$, and thus characterizes
independence. Intuitively, the monotone and antimonotone couplings
represent the two extreme forms of dependence. It is therefore natural
to ask whether they also maximize the Wasserstein distance from the
independent coupling.

For the uniform probability measure $\Leb$ on $[0,1]$, the elements of
$\Pi(\Leb,\Leb)$ are precisely the bivariate copulas. In this setting,
the monotone and antimonotone couplings are supported on the diagonal
and antidiagonal of the unit square, respectively. This led Catalano
and Lavenant to formulate the following conjecture, which remains open
to the best of our knowledge \cite[Remark~3]{CaLa25}.

\begin{conjecture}[{\cite[Remark~3]{CaLa25}}]
Let $\Leb$ denote the uniform probability measure on $[0,1]$. Then
the diagonal and antidiagonal couplings
\[
  (\mathrm{id},\mathrm{id})_\#\Leb
  \qquad\text{and}\qquad
  (\mathrm{id},1-\mathrm{id})_\#\Leb
\]
maximize the 2-Wasserstein distance from the independent
coupling $\Leb \otimes \Leb$ among all elements of
$\Pi(\Leb,\Leb)$.
\end{conjecture}

\subsection{Gaussian result}
We establish the analogous result when $\mu=\gamma$ is the one-dimensional standard Gaussian measure: the diagonal and antidiagonal Gaussian couplings $\Law(Z,Z)$ and $\Law(Z,-Z)$, where $Z\sim\gamma$, maximize the Wasserstein distance from $\gamma_2=\gamma\otimes\gamma$ over $\Pi(\gamma,\gamma)$.

This is the case $n=2$ of the following $n$-marginal result. Let $\gamma_n$ denote the $n$-dimensional standard Gaussian measure and, for $\mu\in\Pcal_2(\R)$, write
\[
  \Pi_n(\mu)
  :=
  \bigl\{
    \pi\in\Pcal_2(\R^n):
    (\mathrm{pr}_i)_\#\pi=\mu,\quad i=1,\ldots,n
  \bigr\}.
\]

\begin{theorem}\label{thm:gaussian-ot}
For every $n\in \mathbb N$, we have
\[
  \sup_{\pi\in\Pi_n(\gamma)}
  \W_2^2(\gamma_n,\pi)
  =2n-2\sqrt n.
\]
The maximizers are precisely the signed diagonal Gaussian couplings
\[
  \pi_\varepsilon
  =\Law(\varepsilon_1Z,\ldots,\varepsilon_nZ), \qquad Z\sim\gamma,
  \qquad
  \varepsilon\in\{-1,1\}^n.
\]

\end{theorem}

A key point of \Cref{thm:gaussian-ot} is that the optimization is taken
over the entire class $\Pi_n(\gamma)$, without any assumption that the
joint law is Gaussian. The theorem shows that allowing non-Gaussian
couplings neither increases the maximal value nor produces additional
maximizers. Mordant and Segers previously solved the corresponding
covariance-matrix problem for two possibly vector-valued Gaussian
marginals, which amounts to restricting the optimization to jointly
Gaussian couplings; see
\cite[Proposition~3.9 and Example~3.14]{MoSe22}. They explicitly left
open whether the same conclusion holds when arbitrary couplings are
allowed.

\subsection{Extension to general marginals}

Although our argument relies crucially on the Gaussian reference
measure $\gamma_n$, the prescribed one-dimensional marginals need not
be Gaussian. More precisely, \Cref{thm:general} below determines, for every
$\mu\in\Pcal_2(\R)$, the maximal distance from $\gamma_n$ among all
elements of $\Pi_n(\mu)$ and characterizes all maximizers.

Let $m_\mu$ denote the mean of $\mu$. If $\mu$ is not symmetric about
$m_\mu$, the diagonal coupling is the unique maximizer. If $\mu$ is
symmetric about $m_\mu$, the maximizers are precisely the signed
diagonal couplings
\[
  \Law\bigl(
    m_\mu+\varepsilon_1(Z-m_\mu),\ldots,
    m_\mu+\varepsilon_n(Z-m_\mu)
  \bigr),
  \qquad
  Z\sim\mu,\quad
  \varepsilon\in\{-1,1\}^n.
\]

\subsection{Related literature}
Wasserstein measures of dependence based on the distance from the
product law were studied in
\cite{MoSz20,NiStMu25,OzLyBeVaLeSe19}; Wiesel considered the
complementary approach through conditional distributions
\cite{Wi22}; see also the survey \cite{CaLa25}. The closest precursor
is the covariance-based Gaussian approach of Mordant and Segers, which
De Keyser and Gijbels extended to several groups
\cite{DeKeGi25,MoSe22}.

\subsection{Organization}
In \Cref{sec:formal}, we present the proof strategy in the Gaussian
case under formal smoothness assumptions. In \Cref{sec:bv}, we collect
the required facts about distributional Hessians, functions of bounded
variation, the Gaussian coarea formula, and Gaussian isoperimetry.
Finally, in \Cref{sec:main}, we prove the general result and derive
\Cref{thm:gaussian-ot} as a special case.

\section{Outline of the Gaussian argument}
\label{sec:formal}
In this section, we give a formal argument for \Cref{thm:gaussian-ot} that highlights the main ideas; all analytic and measure-theoretic details are deferred to the subsequent sections. First, we rephrase the problem in terms of estimating the Laplacian of Brenier potentials. Let $\pi\in\Pi_n(\gamma)$ and $T=\nabla\varphi$ be the Brenier map from $\gamma_n$ to $\pi$. Using $|x-y|^2=|x|^2+|y|^2-2\langle x,y\rangle$, the fact that $\pi$ and $\gamma_n$ both have second moment $n$, and Gaussian integration by parts, we find
\begin{equation}\label{eq:formal-cost}
  \W_2^2(\gamma_n,\pi)
  =2n-2\int_{\R^n}\ip{x}{T(x)}\,d\gamma_n(x) = 2n-2\int_{\R^n}\Delta\varphi\,d\gamma_n(x).
\end{equation}
Therefore, \Cref{thm:gaussian-ot} is equivalent to showing, for every convex function $\varphi$ whose partial derivatives have standard Gaussian law, that
\begin{equation}\label{eq:Laplace_ineq}
\int_{\R^n}\Delta\varphi\,d\gamma_n(x) \ge \sqrt n    
\end{equation}
with equality if and only if $\varphi$ is of the form
\begin{equation}
  \varphi(x)
  =\frac{1}{2\sqrt n}\ip{\varepsilon}{x}^2+c,
  \qquad
  \varepsilon\in\{-1,1\}^n,
  \quad c\in\R.
\end{equation}

\subsection{Proof of the inequality}
In order to prove the inequality \eqref{eq:Laplace_ineq}, we need to use information on first derivatives of $\varphi$ provided by the pushforward condition $(\partial_i\varphi)_\#\gamma_n=\gamma$ to derive a lower bound on an integral of second order derivatives. The natural bridge between these two quantities is a Poincaré-type geometric inequality, applied to $T_i = \partial_i\varphi$. In the present setting, the relevant estimate is obtained from the coarea formula and Gaussian isoperimetry for level sets of $T_i$.

\medskip
We write $\Perg(E)$ for the Gaussian perimeter of $E \subset \R^n$ and note that the perimeter of the half-space $H_t:=\{x\in\R^n:x_1>t\}$ is given by $\Perg(H_t)=\rho_1(t)$, where $\rho_1$ is the standard normal density on $\R$. As $(T_i)_\#\gamma_n=\gamma$, we have
\[
    \gamma_n(\{T_i>t\})
    =
    \gamma((t,\infty))
    =
    \gamma_n(H_t).
\]
The Gaussian isoperimetric inequality states that among sets of prescribed Gaussian measure, half-spaces have minimal Gaussian perimeter. Hence,
\begin{equation}\label{eq:level-set-perimeter}
    \Perg(\{T_i>t\})
    \geq
    \Perg(H_t)
    =
    \rho_1(t).
\end{equation}

The coarea formula identifies the total variation of $T_i$ with the
integral of the perimeters of its superlevel sets. Integrating
\eqref{eq:level-set-perimeter} over $t$ therefore yields
\begin{align}
    \int_{\R^n}|\nabla T_i|\,d\gamma_n
    =
    \int_{\R}\Perg(\{T_i>t\})\,dt 
    \geq
    \int_{\R}\rho_1(t)\,dt
    =
    1.
    \label{eq:formal-coarea}
\end{align}

\medskip

As $\varphi$ is convex, its Hessian $\nabla^2\varphi$ is symmetric positive semidefinite. This allows us to bound $\Delta\varphi = \tr(\nabla^2 \varphi)$ in terms of $|\nabla T_i| = |\nabla^2 \varphi e_i|$, where $e_i$ is the $i$-th unit vector in $\R^n$, using the matrix inequality
\begin{equation}\label{eq:formal-matrix}
    \frac{1}{\sqrt n}\sum_{i=1}^n |Ae_i|
    \leq
    \|A\|_{\mathrm{HS}}
    \leq
    \tr(A).
\end{equation}
Here, $\|A\|_{\mathrm{HS}}$ is the Hilbert--Schmidt norm; see \Cref{lem:matrix} for its proof. Combining the pointwise application of \eqref{eq:formal-matrix} with $A=\nabla^2\varphi$ and \eqref{eq:formal-coarea} yields
\[
\int_{\R^n}\Delta\varphi\,d\gamma_n(x) \ge \frac{1}{\sqrt n}\sum_{i=1}^n \int_{\R^n}|\nabla T_i|\,d\gamma_n
    \ge \sqrt{n}.
\]

\subsection{Characterization of equality cases}

Suppose that a convex function $\varphi$, each of whose partial derivatives has standard Gaussian law, satisfies
\begin{equation}\label{eq:laplacian-equality}
    \int_{\R^n}\Delta\varphi\,d\gamma_n
    =
    \sqrt n.
\end{equation}
Then equality must hold at every step of the preceding chain of
integrated inequalities. We first show equality in the application of the Gaussian isoperimetric inequality forces $\varphi$ to be a quadratic function. To that end, fix $i\in\{1,\ldots,n\}$ and note that we have
\[
    \int_{\R^n}|\nabla T_i|\,d\gamma_n=1.
\]
Equality in the integrated estimate implies
that equality in \eqref{eq:level-set-perimeter} holds for  every
$t\in\R$. Hence, for every $t$, the superlevel set
$\{T_i>t\}$ is a
half-space. Since $(T_i)_\#\gamma_n=\gamma$, this half-space has the
same Gaussian measure as $H_t$. 
Consequently, there is a unit vector $a_i(t)\in\R^n$ such that
\[
    \{T_i>t\}
    =
    \{x\in\R^n:\langle a_i(t),x\rangle>t\}.
\]
The superlevel sets of $T_i$ are nested. Since nested nontrivial
half-spaces must have parallel boundaries and the same orientation,
the vectors $a_i(t)$ coincide. Denoting their common value by $a_i$, we obtain
$
    T_i(x)=\langle a_i,x\rangle.
$

Writing $A$ for the matrix whose $i$-th row is $a_i^\top$, we have 
$
    T(x)=Ax.
$
Since $T=\nabla\varphi$ and $\varphi$ is convex, $A$ is symmetric and
positive semidefinite,
\[
    \varphi(x)
    =
    \frac12\langle Ax,x\rangle+c
\]
for some $c\in\R$.

It remains to determine $A$. Equality in
\eqref{eq:formal-matrix} first gives
$
    \tr(A)=\|A\|_{\mathrm{HS}}.
$
Since the eigenvalues of $A$ are nonnegative, equality between their
$\ell^1$- and $\ell^2$-norms implies that at most one eigenvalue is non-zero. Hence, there is $v\in\R^n$ such that 
$
    A=vv^\top. 
$
Then equality in $\frac{1}{\sqrt n}\sum_{i=1}^n |Ae_i|
    \leq
    \|A\|_{\mathrm{HS}}$
further yields that 
$
    A= t \varepsilon\varepsilon^\top
$
for some $t>0$. The condition $(T_1)_\#\gamma_n=\gamma$ yields that $t =\frac1{\sqrt n}$. 
Consequently,
\[
    \varphi(x)
    =
    \frac1{2\sqrt n}\langle\varepsilon,x\rangle^2+c.
\]

\section{Preliminaries from geometric measure theory}
\label{sec:bv}
The aim of this section is to provide the measure-theoretic background needed to make the proof precise without any regularity assumptions on the convex potentials $\varphi$.

We write $\mathcal L^n$ for the Lebesgue measure on $\R^n$ and $\rho_n$ for the density of $\gamma_n$ w.r.t.\ $\mathcal L^n$. 

\subsection{Distributional derivatives of convex functions}

Let $\varphi:\mathbb R^n\to\mathbb R$ be convex. We write $D\varphi$ for its distributional gradient, $D^2\varphi$ for its distributional Hessian and denote the respective components by $D_i\varphi$ and $D_{ij}\varphi$ for $i,j \in \{1,\dots,n\}$. Crucially, $\varphi$ is
locally Lipschitz and hence the classical gradient $\nabla\varphi$ exists $\mathcal L^n$-almost
everywhere and $ D\varphi=\nabla\varphi\,\mathcal L^n$. 
%  see, for instance, \cite[Section 6.3]{EvansGariepy}. 

The distributional Hessian $D^2\varphi$  is a symmetric positive semidefinite matrix-valued locally finite Radon
measure. Here, positive semidefiniteness means that
$
    \xi^\top D^2\varphi\,\xi
$
is a nonnegative Radon measure for every $\xi\in\R^n$. In
particular, the distributional Laplacian
\[
    \Delta\varphi
    :=
    \operatorname{tr}(D^2\varphi)
    =
    \sum_{i=1}^n D_{ii}\varphi
\]
is a nonnegative locally finite Radon measure.

%As a symmetric positive semidefinite matrix is zero if its trace is zero, $D^2\varphi$ is absolutely continuous w.r.t.\ $\Delta\varphi$. Therefore Radon-Nikodym yields that there is a Borel function $A : \R^n \to \R^{n \times n}$ such that $A(x)$ is symmetric positive semidefinite with $\tr(A(x))=1$ for every $x \in \R^n$ and 
%\begin{align*}
%    D^2 \varphi = A \Delta \varphi. 
%\end{align*}

\begin{remark}
Note that $D^2\varphi$ always refers to the
distributional Hessian, which should not be confused
with the Alexandrov Hessian $\nabla_A^2\varphi$. The latter  exists
$\mathcal L^n$-almost everywhere and is the density of the absolutely
continuous part of $D^2\varphi$. In particular, $\Delta\varphi$ refers to the trace of the distributional Hessian and not the Alexandrov Hessian.    
\end{remark}

\begin{lemma}[Gaussian integration by parts for convex functions]\label{lem:GaussPartInt}
Let $\varphi \in L^1(\gamma_n)$ be convex and assume that
$\nabla\varphi\in L^2(\gamma_n;\mathbb R^n)$. Then
\[
    \int_{\mathbb R^n} \rho_n(x)\,d(\Delta\varphi)(x)
    =
    \int_{\mathbb R^n}
        \langle x,\nabla\varphi(x)\rangle\,d\gamma_n(x)
\]
and in particular, both sides of the equality are finite. 
\end{lemma}
\begin{proof}
Choose $\eta\in C_c^\infty(\R^n)$ such that $0\leq\eta\leq1$,
$\eta=1$ on $B_1$, and $\eta=0$ outside $B_2$, and set
$\eta_R(x):=\eta(x/R)$. By the definition of the distributional
derivative,
\begin{align*}
  \int_{\R^n}\eta_R\rho_n\,d(D_{ii}\varphi)
  =-\int_{\R^n}\partial_i(\eta_R\rho_n)\,
    \partial_i\varphi\,dx
  =\int_{\R^n}\eta_R x_i\partial_i\varphi\,d\gamma_n
    -\int_{\R^n}\partial_i\eta_R\,\partial_i\varphi\,d\gamma_n.
\end{align*}
The first term converges by dominated convergence, since
$x_i\partial_i\varphi\in L^1(\gamma_n)$ by Cauchy--Schwarz, while the
absolute value of the second is bounded by
$C R^{-1}\|\partial_i\varphi\|_{L^2(\gamma_n)}$ and therefore tends to
zero. Since $D_{ii}\varphi$ is a positive measure, Fatou's lemma first
shows that $\int\rho_n\,d(D_{ii}\varphi)<\infty$. Dominated convergence
with respect to this finite weighted measure then gives
\[
  \int_{\R^n}\rho_n\,d(D_{ii}\varphi)
  =\int_{\R^n}x_i\partial_i\varphi\,d\gamma_n.
\]
Summing over $i$ proves the claim.
\end{proof}

\subsection{Gaussian coarea and isoperimetry}

We briefly recall the notions from $BV$ theory that will be used below;
see \cite[Chapter~3]{AmFuPa00} for background. A function $u\in L^1_{\mathrm{loc}}(\R^n)$ belongs to
$\BVloc(\R^n)$ if its distributional derivative is an
$\R^n$-valued locally finite Radon measure. If $M$ is an
$\R^k$-valued Radon measure on $\R^n$, we write $|M|$ for its total variation measure.

For a Borel set $E\subseteq\R^n$, we write $\ind_E$ for its indicator
function. The set $E$ has locally finite perimeter if
$\ind_E\in BV_{\mathrm{loc}}(\R^n)$, in which case $|D\ind_E|$ is its
perimeter measure. Its Gaussian perimeter is defined by
\[
  \Perg(E)
  :=
  \int_{\R^n}\rho_n\,d|D\ind_E|.
\]
We set $\Perg(E)=+\infty$ if $E$ does not have locally finite
perimeter.

\begin{remark}
We use this measure-theoretic definition because it requires no
regularity of the level sets appearing below. For sets with sufficiently
regular boundary, it agrees with the familiar surface integral:
\[
  |D\ind_E|
  =
  \mathcal H^{n-1}|_{\partial E},
  \qquad
  \Perg(E)
  =
  \int_{\partial E}\rho_n\,d\mathcal H^{n-1}.
\]
If $E$ is bounded with $C^2$ boundary, it also agrees with its Gaussian
outer Minkowski content:
\[
  \Perg(E)
  =
  \lim_{r\downarrow0}
  \frac{\gamma_n(E_r)-\gamma_n(E)}{r},
  \qquad
  E_r:=\{x\in\R^n:\operatorname{dist}(x,E)<r\}.
\]
\end{remark}

We record the coarea formula in the weighted form used below.

\begin{proposition}[Gaussian coarea formula]
\label{prop:gaussian-coarea}
For every $u\in BV_{\mathrm{loc}}(\R^n)$,
\[
  \int_{\R^n}\rho_n\,d|Du|
  =
  \int_\R\Perg(\{u>t\})\,dt,
\]
with equality in $[0,+\infty]$.
\end{proposition}

This is the standard measure-valued coarea formula
\cite[Theorem~3.40 and (3.63)]{AmFuPa00}, integrated against the
nonnegative weight $\rho_n$.

We next recall the geometric inequality that will be applied to the
superlevel sets in \Cref{prop:gaussian-coarea}. For a unit vector
$a\in\R^n$ and $r\in[-\infty,+\infty]$, let
\begin{equation}\label{eq:gaussian-halfspaces}
  H_{a,r}:=\{x\in\R^n:\langle a,x\rangle>r\}.
\end{equation}
For convenience, this notation includes the entire space $H_{a,-\infty}=\R^n$ and the empty set $H_{a,+\infty}=\emptyset$. We write $\Phi$ for the cdf of the standard Gaussian on $\R$, $\rho_1$ for its density, $\Phi^{-1}$ for the quantile function and use the conventions
\[
\Phi(-\infty)=0, 
\qquad
\Phi(+\infty)=1,
\qquad
  \Phi^{-1}(0)=-\infty,
  \qquad
  \Phi^{-1}(1)=+\infty,
  \qquad
  \rho_1(\pm\infty)=0.
\]
By rotational invariance of $\gamma_n$,
\[
  \gamma_n(H_{a,r})=1-\Phi(r),
  \qquad
  \Perg(H_{a,r})=\rho_1(r).
\]
Consequently, every half-space of Gaussian measure $s\in[0,1]$ has
Gaussian perimeter
\[
  I_\gamma(s):=\rho_1(\Phi^{-1}(s)),
\]
where we use the symmetry of $\rho_1$. The Gaussian isoperimetric
inequality states that half-spaces minimize Gaussian perimeter among
all sets of prescribed Gaussian measure. The inequality was proved independently by Sudakov--Tsirel'son \cite{SuTs78} and Borell \cite{Bo75}. For the equality cases in the form used
here, see Carlen and Kerce \cite[Theorem~2]{CaKe01}.

\begin{theorem}[Gaussian isoperimetric inequality]
\label{thm:gaussian-isoperimetry}
Let $E\subseteq\R^n$ be a Borel set. Then
\[
  \Perg(E)
  \geq
  I_\gamma(\gamma_n(E)).
\]
Equality holds if and only if $E$ agrees $\gamma_n$-almost everywhere
with a half-space.
\end{theorem}

\section{General result and proof}\label{sec:main}

\begin{theorem}\label{thm:general}
Let $n\geq2$ and $\mu\in\Pcal_2(\R)$, and denote the mean of $\mu$ by
$m_\mu$. Then
\begin{equation}\label{eq:general-value}
  \sup_{\pi\in\Pi_n(\mu)}
  \W_2^2(\gamma_n,\pi)
  =
  (n-1)
  +
  \W_2^2\bigl(\gamma,(\sqrt n\,\mathrm{id})_\#\mu\bigr).
\end{equation}
The maximizers are classified as follows.
\begin{enumerate}[label=\textup{(\roman*)}]
  \item If $\mu$ is not symmetric about $m_\mu$, the unique maximizer is
  the diagonal coupling
  \[
    \pi_{\Delta,\mu}
    :=
    \Law(Z,\ldots,Z),
    \qquad Z\sim\mu.
  \]

  \item If $\mu$ is symmetric about $m_\mu$, the maximizers are precisely
  the signed diagonal couplings
  \begin{equation}\label{eq:signed-mu-diagonals}
    \pi_{\varepsilon,\mu}
    :=
    \Law\bigl(
      m_\mu+\varepsilon_1(Z-m_\mu),\ldots,
      m_\mu+\varepsilon_n(Z-m_\mu)
    \bigr),
    \qquad
    \varepsilon\in\{-1,1\}^n,
  \end{equation}
  where $Z\sim\mu$.
\end{enumerate}
\end{theorem}

The proof has two main ingredients. Gaussian isoperimetry gives a
sharp lower bound on the Gaussian total variation of each component
of a Brenier map. An elementary matrix inequality then combines these
scalar estimates into a lower bound for the distributional Laplacian
of its convex potential.

We first introduce the notation used below. For
$\alpha,\beta\in\Pcal_2(\R^k)$, let $\Pi(\alpha,\beta)$ denote the set
of their couplings and set
\[
  \MC(\alpha,\beta)
  :=
  \sup_{\kappa\in\Pi(\alpha,\beta)}
  \int_{\R^k\times\R^k}\langle x,y\rangle\,d\kappa(x,y),
  \qquad
  \mtwo(\alpha)
  :=
  \int_{\R^k}|x|^2\,d\alpha(x).
\]
Then
\[
  \W_2^2(\alpha,\beta)
  =
  \mtwo(\alpha)+\mtwo(\beta)-2\MC(\alpha,\beta).
\]
Writing $\Phi$ for the distribution function of $\gamma$, define
\[
  F_\mu(t):=\mu((-\infty,t]),
  \qquad
  F_\mu^{-1}(s):=\inf\{y\in\R:F_\mu(y)\geq s\},
\]
and let
\begin{equation}\label{eq:qmu-general}
  \qmu(t):=F_\mu^{-1}(\Phi(t)),
  \qquad t\in\R.
\end{equation}
The map $\qmu$ is the monotone rearrangement from $\gamma$ to $\mu$.

The first ingredient is the following functional consequence of
Gaussian isoperimetry. It will later be applied to the components of
a Brenier map.

\begin{proposition}\label{prop:ineq_component}
Let $\mu \in \Pcal_2(\R)$ and $u\in \BVloc(\R^n)$ satisfy
$u_\#\gamma_n=\mu$. Then
\begin{equation}\label{eq:scalar-bound}
  \int_{\R^n} \rho_n \,  d|Du| \ge \MC(\gamma,\mu).
\end{equation}
Equality holds if and only if there is a
unit vector $a\in\R^n$ such that
\begin{equation}\label{eq:scalar-equality}
  u(x)=\qmu(\ip{a}{x})
  \qquad\text{for $\gamma_n$-almost every $x$.}
\end{equation}
\end{proposition}

\begin{proof}
Set $E_t:=\{u>t\}$. The pushforward constraint gives
\[
  \gamma_n(E_t)=\mu((t,\infty))
  \qquad\text{for every }t\in\R.
\]
Hence, by the Gaussian coarea formula and Gaussian isoperimetry,
\begin{align}\label{eq:scalar-profile}
  \int_{\R^n}\rho_n\,d|Du|
  =\int_\R \Perg(E_t)\,dt
  \geq
  \int_\R I_\gamma\bigl(\mu((t,\infty))\bigr)\,dt.
\end{align}
For every $t\in\R$, we have
\[
  \{\qmu>t\}
  =
  \bigl(\Phi^{-1}(F_\mu(t)),\infty\bigr).
\]
Since $\rho_1'(s)=-s\rho_1(s)$ and
$\mu((t,\infty))=1-F_\mu(t)$, it follows that
\begin{align*}
  \int_{\{\qmu>t\}}s\,d\gamma(s)
  =
  \rho_1\bigl(\Phi^{-1}(F_\mu(t))\bigr)
  =
  I_\gamma\bigl(\mu((t,\infty))\bigr).
\end{align*}
Integrating this identity in $t$, using
$\int_\R s\,d\gamma(s)=0$, and applying Fubini's theorem gives
\begin{align*}
  \int_\R I_\gamma\bigl(\mu((t,\infty))\bigr)\,dt
  &=
  \int_\R\int_\R
  s\bigl(
    \mathbf 1_{\{\qmu(s)>t\}}
    -\mathbf 1_{\{0>t\}}
  \bigr)\,dt\,d\gamma(s) \\
  &=
  \int_\R s\,\qmu(s)\,d\gamma(s).
\end{align*}
The last integral equals $\MC(\gamma,\mu)$ by one-dimensional
monotone rearrangement. This proves \eqref{eq:scalar-bound}. It is
straightforward to check that for every $a\in\mathbb S^{n-1}$, the
function $u(x)=\qmu(\langle a,x\rangle)$ is an equality case.

Suppose now that equality holds. With the notation from
\Cref{eq:gaussian-halfspaces}, set
\[
  r_t:=\Phi^{-1}(F_\mu(t))\in[-\infty,+\infty].
\]
Then, for every $a\in\mathbb S^{n-1}$ and $t \in \R$
\[
  \gamma_n(H_{a,r_t})
  =
  1-\Phi(r_t)
  =
  \mu((t,\infty))
  =
  \gamma_n(E_t).
\]
As $\Perg(E_t)
  \ge
  I_\gamma\bigl(\gamma_n(E_t)\bigr)$ for every $t \in \R$, equality in \eqref{eq:scalar-profile} yields that for almost every $t \in \R$
\[
\Perg(E_t)
  =
  I_\gamma\bigl(\gamma_n(E_t)\bigr).
\]
For such $t$, the equality cases in Gaussian isoperimetry yield
$a_t\in\mathbb S^{n-1}$ such that
\[
  \gamma_n\bigl(E_t\mathbin{\triangle}H_{a_t,r_t}\bigr)=0.
\]

Let $s<t$ be two such levels with $F_\mu(s),F_\mu(t) \in (0,1)$. Then $E_t\subseteq E_s$ and hence $
  \gamma_n\bigl(H_{a_t,r_t}\setminus H_{a_s,r_s}\bigr)=0.
$ As two nontrivial half-spaces that are nested up to a Gaussian null set
have the same oriented normal, it follows that all the vectors $a_t$
coincide with some $a\in\mathbb S^{n-1}$. (If there is no $t$ such that $r_t$ is finite, we choose $a$ arbitrary). Therefore, we have for almost all $t$,
\[
\gamma_n\bigl(E_t\mathbin{\triangle}H_{a,r_t}\bigr)=0.
\]
Set
\[
  v(x):=\qmu(\langle a,x\rangle).
\]
By the definition of the quantile, we have
$  \{v>t\}
  =
  H_{a,r_t}
$. Hence, for almost every $t\in\R$
\[
  \gamma_n\bigl(
    \{u>t\}\mathbin{\triangle}\{v>t\}
  \bigr)=0.
\]
As $
  |r-s|
  =
  \int_\R
  \left|
    \mathbf 1_{\{r>t\}}
    -
    \mathbf 1_{\{s>t\}}
  \right|\,dt
$, we find using Fubini's theorem
\begin{align*}
  \int_{\R^n}|u-v|\,d\gamma_n
  &=
  \int_\R
  \gamma_n\bigl(
    \{u>t\}\mathbin{\triangle}\{v>t\}
  \bigr)\,dt
  =0.
\end{align*}
Thus $u(x)=\qmu(\langle a,x\rangle)$ for $\gamma_n$-almost every $x$.
\end{proof}

\begin{remark}
\Cref{prop:ineq_component} can be seen as a functional version of the Gaussian isoperimetric inequality. In particular, taking $u=\ind_E$ recovers the
Gaussian isoperimetric inequality. Indeed, writing
$p:=\gamma_n(E)$ we have 
$
  (\ind_E)_\#\gamma_n
  =
  (1-p)\delta_0+p\delta_1
  =:\mu_E
$
and one-dimensional monotone rearrangement gives
\[
  \MC(\gamma,\mu_E)
  =
  \int_{\Phi^{-1}(1-p)}^\infty s\,d\gamma(s)
  =
  \rho_1\bigl(\Phi^{-1}(1-p)\bigr)
  =
  I_\gamma(p).
\]
Thus \eqref{eq:scalar-bound} becomes
$
  \Perg(E)\geq I_\gamma(\gamma_n(E))
$
and \eqref{eq:scalar-equality} states that the equality cases are precisely the indicator functions of half-spaces.
\end{remark}

\begin{lemma}\label{lem:matrix}
Let $A\in\mathbb R^{n\times n}$ be symmetric and positive semidefinite
and write $e_i$ for the $i$-th unit vector. Then
\[
  \frac{1}{\sqrt n}\sum_{i=1}^n |Ae_i|
  \leq
  \operatorname{tr}(A).
\]
Equality holds if and only if
\[
  A=c\,\varepsilon\varepsilon^\top
\]
for some $c\geq0$ and some $\varepsilon\in\{-1,1\}^n$.
\end{lemma}

\begin{proof}
The Hilbert--Schmidt norm satisfies
\[
  \|A\|_{\mathrm{HS}}^2
  =
  \sum_{i=1}^n|Au_i|^2
\]
for every orthonormal basis $(u_1,\ldots,u_n)$. Applying this identity
first to the standard basis and then to an orthonormal eigenbasis of
$A$, we obtain
\[
  \sum_{i=1}^n|Ae_i|^2
  =
  \sum_{i=1}^n\lambda_i^2,
\]
where $\lambda_1,\ldots,\lambda_n$ are the eigenvalues of $A$.
Cauchy--Schwarz and the nonnegativity of the eigenvalues give
\[
  \frac1{\sqrt n}\sum_{i=1}^n|Ae_i|
  \leq
  \left(\sum_{i=1}^n|Ae_i|^2\right)^{1/2}
  =
  \left(\sum_{i=1}^n\lambda_i^2\right)^{1/2}
  \leq
  \sum_{i=1}^n\lambda_i
  =
  \operatorname{tr}(A).
\]

Suppose that equality holds. Equality in the second inequality forces
$A$ to have rank at most one. If $A\neq0$, we may therefore write
$A=\lambda vv^\top$ for some $\lambda>0$ and some unit vector $v$.
Equality in the Cauchy--Schwarz inequality gives
$|Ae_i|=|Ae_j|$ for all $i,j$, and hence
$|v_i|=1/\sqrt n$ for every $i$. Thus
$v=\varepsilon/\sqrt n$ for some
$\varepsilon\in\{-1,1\}^n$, and
\[
  A=\frac{\lambda}{n}\varepsilon\varepsilon^\top.
\]
The case $A=0$ corresponds to $c=0$. Conversely, a direct computation
shows that every matrix $A=c\,\varepsilon\varepsilon^\top$ satisfies
the inequality with equality.
\end{proof}

The following lemma will be crucial in the classification of equality cases in the proof of the main result.

\begin{lemma}\label{lem:signed-gradient}
Let $\mu\in\Pcal_2(\R)$ and let
$\varepsilon\in\{-1,1\}^n$. Define
$T_\varepsilon\colon\R^n\to\R^n$ by
\[
  T_{\varepsilon,i}(x)
  :=
  \qmu\left(
    \frac{\varepsilon_i}{\sqrt n}
    \langle\varepsilon,x\rangle
  \right),
  \qquad i\in\{1,\ldots,n\},
\]
and define 
\[
  \varphi_\varepsilon(x)
  :=
  m_\mu\sum_{i=1}^n x_i
  +
  \sqrt n\,h_\mu\left(
    \frac{\langle\varepsilon,x\rangle}{\sqrt n}
  \right),
  \qquad
  h_\mu(s)
  :=
  \int_0^s\bigl(\qmu(r)-m_\mu\bigr)\,dr.
\]
Then $\varphi_\varepsilon$ is convex and the following statements hold:
\begin{itemize}
    \item If $\varepsilon=(1,\ldots,1)$ or $\varepsilon=(-1,\ldots,-1)$, then $T_\varepsilon=\nabla\varphi_{(1,\ldots,1)}$.
    \item If $\varepsilon$ contains both signs, then $T_\varepsilon$ is the
gradient of a convex function if and only if $\mu$ is symmetric about
$m_\mu$. In this case, $T_\varepsilon=\nabla\varphi_\varepsilon$.
\end{itemize}
\end{lemma}

\begin{proof}
Since $\qmu$ is nondecreasing, $h_\mu$ and hence
$\varphi_\varepsilon$ are convex. Moreover, $\varphi_\varepsilon$ is $\mathcal L^n$-almost everywhere differentiable, and
\begin{equation}\label{eq:gradient-signed-potential}
  \partial_i\varphi_\varepsilon(x)
  =
  m_\mu
  +
  \varepsilon_i\left(
    \qmu\left(\frac{\langle\varepsilon,x\rangle}{\sqrt n}\right)
    -m_\mu
  \right).
\end{equation}
If $\varepsilon=(1,\ldots,1)$, this equals
$T_{\varepsilon,i}(x)$. Since
$T_{-\varepsilon}=T_\varepsilon$, this also proves the first claim
when $\varepsilon=(-1,\ldots,-1)$.

Suppose now that $\varepsilon$ contains both signs and that
$T_\varepsilon=\nabla\varphi$ for some convex function $\varphi$. We
identify $\R\times\varepsilon^\perp$ with $\R^n$ through
$(s,y)\mapsto s\varepsilon/\sqrt n+y$ and write
$\mathcal L_{\varepsilon^\perp}$ for Lebesgue measure on
$\varepsilon^\perp$. Then
\[
  DT_{\varepsilon,i}
  =
  \frac{\varepsilon}{\sqrt n}
  \left(
    D\bigl(q_\mu(\varepsilon_i\,\cdot)\bigr)
    \otimes\mathcal L_{\varepsilon^\perp}
  \right),
\]
The symmetry of $D^2\varphi=DT_\varepsilon$ therefore yields, for all $i,j$,
\[
  \varepsilon_j
  D\bigl(q_\mu(\varepsilon_i\,\cdot)\bigr)
  =
  \varepsilon_i
  D\bigl(q_\mu(\varepsilon_j\,\cdot)\bigr)
\]
as Radon measures on $\R$.
Choosing $i,j$ such that $\varepsilon_i=1$ and
$\varepsilon_j=-1$, we obtain
\[
  D\bigl(q_\mu(\,\cdot\,)+q_\mu(-\,\cdot\,)\bigr)=0.
\]
Thus $s\mapsto q_\mu(s)+q_\mu(-s)$ is constant almost everywhere.
Integration with respect to $\gamma$ shows that this constant is
$2m_\mu$, and hence
\begin{equation}\label{prf:qsym}
  q_\mu(-s)=2m_\mu-q_\mu(s)
  \qquad\text{for almost every }s\in\R.
\end{equation}
Therefore, $\mu$ is symmetric about $m_\mu$.

Conversely, if $\mu$ is symmetric about $m_\mu$, then \eqref{prf:qsym} holds. Consequently, the right-hand side of
\eqref{eq:gradient-signed-potential} equals
$T_{\varepsilon,i}(x)$ for every $i$, hence 
$T_\varepsilon=\nabla\varphi_\varepsilon$.
\end{proof}

We are now ready to prove the main result.

\begin{proof}[Proof of \Cref{thm:general}]
\emph{Upper bound.}
Fix $\pi\in\Pi_n(\mu)$ and let $T=\nabla\varphi$, where $\varphi \in L^1(\gamma_n)$ is convex, be the Brenier map
from $\gamma_n$ to $\pi$. Note that 
\[
  \mtwo(\gamma_n)=n,
  \qquad
  \mtwo(\pi)=n\,\mtwo(\mu),
\]
and $\nabla\varphi\in L^2(\gamma_n;\R^n)$ because $\| \nabla\varphi\|^2_{L_2(\gamma_n)} =   \mtwo(\pi) < \infty$. Thus Gaussian
integration by parts, \Cref{lem:GaussPartInt}, gives
\begin{align}\label{eq:wasserstein-laplacian}
  \W_2^2(\gamma_n,\pi)
  &=
  n+n\,\mtwo(\mu)
  -2\MC(\gamma_n,\pi) \notag\\
  &=
  n+n\,\mtwo(\mu)
  -2\int_{\R^n} \langle \nabla\varphi(x),x \rangle \, d\gamma_n(x) \notag \\
  &=
  n+n\,\mtwo(\mu)
  -2\int_{\R^n}\rho_n\,d(\Delta\varphi).
\end{align}
As a symmetric positive semidefinite matrix is zero if its trace is zero, $D^2\varphi$ is absolutely continuous w.r.t.\ $\Delta\varphi$. By the Radon--Nikodým theorem, there is a Borel function $A : \R^n \to \R^{n \times n}$ such that
\[
D^2 \varphi = A \Delta \varphi, \qquad \tr(A(x))=1, \qquad A(x) \text{ symmetric positive semidefinite}.
\]
For $T_i=\partial_i\varphi$, we have $T_i\in \BVloc(\R^n)$
and
\[
  D T_i=D^2\varphi\,e_i=(Ae_i)\,\Delta\varphi,
  \qquad
  d|D T_i|=|Ae_i|\,d(\Delta\varphi).
\]
Applying \Cref{lem:matrix} pointwise to $A$ and then
\Cref{prop:ineq_component} to each $T_i$ yields
\begin{align}\label{eq:general-laplacian-bound}
  \int_{\R^n}\rho_n\,d(\Delta\varphi)
  \geq
  \frac1{\sqrt n}\sum_{i=1}^n
  \int_{\R^n}\rho_n\,d|D T_i| 
  \geq
  \frac1{\sqrt n}\sum_{i=1}^n\MC(\gamma,\mu)
  =
  \sqrt n\,\MC(\gamma,\mu).
\end{align}
Combining \eqref{eq:wasserstein-laplacian} and
\eqref{eq:general-laplacian-bound}, we obtain
\[
  \W_2^2(\gamma_n,\pi)
  \leq
  n+n\,\mtwo(\mu)-2\sqrt n\,\MC(\gamma,\mu).
\]
Using that 
\[
  \W_2^2\bigl(\gamma,(\sqrt n\,\mathrm{id})_\#\mu\bigr)
  =
  1+n\,\mtwo(\mu)-2\sqrt n\,\MC(\gamma,\mu),
\]
we conclude 
\[
\W_2^2(\gamma_n,\pi) 
  \leq \W_2^2\bigl(\gamma,(\sqrt n\,\mathrm{id})_\#\mu\bigr) + n-1.
\]

\medskip

\noindent\emph{Construction of maximizers.}
Take $\varepsilon=(1,\ldots,1)$; if $\mu$ is symmetric about
$m_\mu$, let more generally $\varepsilon\in\{-1,1\}^n$. By
\Cref{lem:signed-gradient}, $T_\varepsilon$ is the gradient of a
convex function and is therefore the optimal transport from
$\gamma_n$ to
\[
  \pi_\varepsilon:=(T_\varepsilon)_\#\gamma_n.
\]
Each component of $T_\varepsilon$ has law $\mu$, so
$\pi_\varepsilon\in\Pi_n(\mu)$. If
$\varepsilon=(1,\ldots,1)$, then $\pi_\varepsilon=\pi_{\Delta,\mu}$.
If $\mu$ is symmetric about $m_\mu$, then
\[
  \qmu(\varepsilon_i s)
  =
  m_\mu+\varepsilon_i\bigl(\qmu(s)-m_\mu\bigr)
\]
for almost every $s$, and hence
$\pi_\varepsilon=\pi_{\varepsilon,\mu}$. Moreover,
\begin{align*}
  \int_{\R^n}
  \langle x,T_\varepsilon(x)\rangle\,d\gamma_n(x)
  &=
  \frac1{\sqrt n}\sum_{i=1}^n
  \int_\R \varepsilon_i s\,\qmu(\varepsilon_i s)\,d\gamma(s)\\
  &=
  \sqrt n\int_\R s\,\qmu(s)\,d\gamma(s)
  =
  \sqrt n\,\MC(\gamma,\mu).
\end{align*}
Thus $\MC(\gamma_n,\pi_\varepsilon)
=\sqrt n\,\MC(\gamma,\mu)$, and
\[
  \W_2^2(\gamma_n,\pi_\varepsilon)
  =
  \W_2^2\bigl(\gamma,(\sqrt n\,\mathrm{id})_\#\mu\bigr)+n-1.
\]

\medskip
\noindent\emph{Characterization of equality cases.}
Let $\pi\in\Pi_n(\mu)$ be a maximizer and let $T=\nabla\varphi$
be the optimal map from $\gamma_n$ to $\pi$. If
$\Delta\varphi=0$, then $T$ is constant, so $\mu$ is a Dirac mass
and the conclusion is immediate. We may therefore assume that
$\Delta\varphi\neq0$.

Equality in \eqref{eq:general-laplacian-bound} implies that each
$T_i$ is an equality case in \Cref{prop:ineq_component}. Hence, for
every $i\in\{1,\ldots,n\}$, there is a unit vector $a_i\in\R^n$
such that
\begin{equation}\label{eq:equality-components}
  T_i(x)=\qmu(\langle a_i,x\rangle)
  \qquad\text{for $\gamma_n$-almost every }x.
\end{equation}

Write $B$ for the Radon--Nikodým derivative of $D^2\varphi$ with
respect to $\Delta\varphi$, so that
\[
  D^2\varphi=B\,\Delta\varphi,
  \qquad
  \operatorname{tr}(B)=1,
  \qquad
  B\text{ is symmetric and positive semidefinite}
\]
for $\Delta\varphi$-almost every $x$. It is straightforward to check that $DT_i = a_i ( D\qmu \otimes \mathcal{L}_{a_i^\bot})$, where $\operatorname{span}(a_i)\times a_i^\bot$ is identified with $\R^n$ through $(s a_i,y)\mapsto s a_i+y$ and $\mathcal{L}_{a_i^\bot}$ is Lebesgue measure on $a_i^\bot$. Since $\qmu$ is nondecreasing, $D \qmu \ge 0$ and hence
\[
  DT_i=a_i\,|DT_i|.
\]
On the other hand,
\[
  DT_i=D^2\varphi\,e_i=(Be_i)\,\Delta\varphi,
  \qquad
  |DT_i|=|Be_i|\,\Delta\varphi,
\]
and therefore
\begin{equation}\label{eq:prf:B_struc}
  Be_i=a_i|Be_i|
  \qquad\text{for $\Delta\varphi$-almost every }x.
\end{equation}

Equality in the first inequality of
\eqref{eq:general-laplacian-bound}, together with $\rho_n>0$, yields
\[
  \frac1{\sqrt n}\sum_{i=1}^n|B(x)e_i|=1
\]
for $\Delta\varphi$-almost every $x$. By the equality cases in
\Cref{lem:matrix}, for such $x$ there is
$\sigma(x)\in\{-1,1\}^n$ such that
\[
  B(x)=\frac1n\sigma(x)\sigma(x)^\top.
\]
In particular,
\[
  |B(x)e_i|=\frac1{\sqrt n},
  \qquad
  a_i=\frac{\sigma_i(x)}{\sqrt n}\sigma(x),
\]
where the second identity follows from \eqref{eq:prf:B_struc}.
Taking $i=1$, we see that
\[
  \varepsilon:=\sqrt n\,a_1=\sigma_1(x)\sigma(x)
\]
is a fixed vector in $\{-1,1\}^n$. Consequently, for every $i\in\{1,\ldots,n\}$, 
\[
  a_i=\frac{\varepsilon_i}{\sqrt n}\varepsilon.
\]
It follows from \eqref{eq:equality-components} that
\[
  T_i(x)
  =
  \qmu\left(
    \frac{\varepsilon_i}{\sqrt n}
    \langle\varepsilon,x\rangle
  \right)
  =
  T_{\varepsilon,i}(x).
\]

We may now apply \Cref{lem:signed-gradient}. If all entries of
$\varepsilon$ have the same sign, then $\pi$ is the diagonal
coupling. If $\varepsilon$ contains both signs, the fact that
$T_\varepsilon$ is the gradient of a convex function forces $\mu$ to
be symmetric about $m_\mu$, and then
$\pi=\pi_{\varepsilon,\mu}$. Together with the construction above,
this proves the asserted classification of all maximizers.
\end{proof}

\begin{proof}[Proof of \Cref{thm:gaussian-ot}]
The claim  follows from \Cref{thm:general} with $\mu=\gamma$ as $\gamma$ is symmetric and  $\W_2(\gamma, (\sqrt{n}\,\textup{id})_\# \gamma) = \sqrt{n}-1$. 
\end{proof}

\section*{Acknowledgments}
The author thanks Hugo Lavenant for bringing the uniform-marginal
problem to his attention and Leo Brauner for helpful discussions. The
author used a large language model for assistance with exposition and
technical checks.
The proof strategy is due to the author, who independently verified all
mathematical arguments. This research was funded in whole or in part by the Austrian Science
Fund (FWF) [10.55776/J4981].
For open access purposes, the author has applied a CC BY public copyright license to any author accepted manuscript version arising from this submission.

\bibliographystyle{plain}
\bibliography{joint_biblio}

\end{document}